\documentclass[12pt]{amsart}
\usepackage[alphabetic]{amsrefs}
\usepackage{amssymb,amsmath,amsthm,mathrsfs,ulem,tikz, caption}
\usepackage[alphabetic]{amsrefs}
\usepackage{amsmath,amssymb,amsthm,enumerate,tikz,graphicx}
\usepackage[all,cmtip]{xy}
\usepackage[applemac]{inputenc}
\usepackage{color}

\newtheorem{thm}{Theorem}[subsection]
\newtheorem{theorem}[thm]{Theorem}
\newtheorem{conj}[thm]{Conjecture}
\newtheorem{lemma}[thm]{Lemma}
\newtheorem{corollary}[thm]{Corollary}
\newtheorem{proposition}[thm]{Proposition}
\numberwithin{equation}{section} \theoremstyle{definition}

\usepackage{graphicx,color}
\usepackage{amsmath}
\usepackage{latexsym}
\usepackage{epsfig}
\usepackage{amsfonts}
\usepackage{amssymb}
\usepackage[margin=2cm]{geometry}

\newcommand{\I}{\mathcal{I}}

\newcommand{\m}{^{\text{-1}}}

\def \Z{\mathbb{Z}}
\def \Ga {\Gamma}
\def \ga {\gamma}

\def \Alt{\mathrm{Alt}}
\def \Stab {\mathrm{Stab}}
\def \tr {\mathrm{tr}}
\def \scrB {\mathscr{B}}
\def \I {I}
\def \scrA {\mathscr{A}}

\def \Z{\mathbb{Z}}

\usetikzlibrary{decorations.markings}

\def \Z{\mathbb{Z}}
\def \Ga {\Gamma}
\def \ga {\gamma}
\def \tr {\mathrm{tr}}
\def \I {I}
\def \scrA {\mathscr{A}}
\def \scrB {\mathscr{B}}
\def \scrP {\mathscr{P}}
\def \scrC {\mathscr{C}}
\def \scrT {\mathscr{T}}
\def \scrL {\mathscr{L}}

\def \Aut {\mathrm{Aut}}
\def \Conv {\mathrm{Conv}}

\def \Alt{\mathrm{Alt}}
\def \Stab {\mathrm{Stab}}
\begin{document}

\pagestyle{myheadings} \markright{Twisted Poincare Series and Zeta functions on finite quotients of buildings}

\title{Twisted Poincare Series and Zeta functions on finite quotients of buildings}
\author{Ming-Hsuan Kang, and Rupert M\MakeLowercase{c}Callum}
\address{Ming-Hsuan Kang\\ Department of Applied Mathematics\\National Chiao-Tung University\\
Hsinchu, Taiwan} \email{\tt mhkang@math.nctu.edu.tw}
\address{Rupert M\MakeLowercase{c}Callum \\ \\ Department of Mathematics\\ University of T\"{u}bingen\\ T\"{u}bingen, Germany}
\email{\tt {rupert-gordon.mccallum@uni-tuebingen.de}}
\thanks{The research of the first author is supported by the MOST grant 104-2115-M-009-006-MY2. {The research of the second author is supported by the DFG grant DE 436/10-1.}}
\date{}
\maketitle

{\bf Abstract:} In the case where $G=$SL$_{2}(F)$ for a non-archimedean local field $F$ and $\Gamma$ is a discrete torsion-free cocompact subgroup of $G$, there is a known relationship between the Ihara zeta function for the quotient of the Bruhat-Tits tree of $G$ by the action of $\Gamma$, and an alternating product of determinants of twisted Poincar\'e series for parabolic subgroups of the affine Weyl group of $G$. We show how this can be generalised to other split simple algebraic groups of rank two over $F$, and formulate a conjecture about how this might be generalised to groups of higher rank. 

\tableofcontents

\section{Introduction}
The classical Ihara zeta function \cite{Iha} is a counting function associated to a discrete torsion-free cocompact subgroup $\Ga$ of $G=$PGL$_2(F)$, where $F$ is a non-archimedean local field with $q$ elements in its residue field, defined as follows:
$$ Z(\Gamma,u) = \prod_{[\ga]}(1-u^{\ell(\ga)})^{-1}.$$
Here $[\ga]$ runs through all primitive conjugacy classes in $\Ga$ and {$\ell$ is the length function, and we always select an element $\ga$ of the conjugacy class which minimises $\ell(\ga)$}. On the other hand, one can also regard  $Z(\Gamma,u)$ as a geometric counting function associated to the quotient of the Bruhat-Tits tree of $G$ by the action of $\Ga$, denoted by $\scrB_\Ga$, such that 
$$ Z(\scrB_\Ga,u) = \prod_{[c]}(1-u^{\ell(c)})^{-1}.$$
Here $[c]$ runs through all equivalence classes of primitive closed geodesic walks in $\scrB_\Ga$ and $\ell(c)$ is the path length of $c$ in graph theory.

Ihara showed that the zeta function is indeed a rational function given by the following formula
\begin{equation} \label{Ih1}
 Z(\scrB_\Ga,u) = \frac{(1-u^2)^{\chi(\scrB_\Ga)}}{\det(1-Au+ qu^2)}.
\end{equation}
Here $\chi(\scrB_\Ga)$ is the Euler characteristic and $A$ is the adjacency matrix of the finite graph $\scrB_\Ga$. 

It is natural to ask if one can generalize Ihara's result to other reductive groups. However, there is no canonical way to define Ihara zeta functions on finite quotients of higher-dimensional buildings.
Thus one must find a new interpretation of Eq.(\ref{Ih1}). 

For example, the term ${\det(1-Au+ qu^2)}^{-1}$ can be regarded as an unramified Langlands $L$-function associated to the representation $L^2(\Ga \backslash G)$ of $G$. 
In this case, the right hand side of Eq.(\ref{Ih1}) is known in general and it remains to figure out the left-hand side in terms of geometric zeta functions. For this viewpoint, see \cite{KL1} and \cite{KL2} for the case of PGL$_3$,
\cite{KLW} for the cases of  PGL$_3$ and  PGSP$_4$, and \cite{DK} for the case of PGL$_n$ over the local field {such that the} residue field is the field with one element.

On the other hand, Hashimoto \cite{Ha1} showed that the Ihara zeta function can be expressed in terms of the edge adjacency operator in an obvious way. In this case, Eq.(\ref{Ih1}) becomes a relation between two parahoric Hecke operators (which are the edge adjacency operator and the vertex adjacency operator). In general, one can obtain a similar identity by comparing eigenvalues of proper parahoric operators using representation theory. See \cite{KL1} and \cite{FLW} for the case of PGL$_3$ and GSP$_4$, respectively, from the view point of this approach.

Besides, when $G=$SL$_2(F)$, let $W$ be the affine Weyl group of $G$ with the standard generating set $S$ of order two. For each parabolic subgroup $W_I$ generating by $I \subset S$, one can consider its Poincar\'{e} series $W_I(\rho,u)$
twisted by the representation $\rho= L^2(\Ga \backslash G)$. In this case, Hashimoto's result implies that
\begin{equation} \label{Ih2} 
 Z(\scrB_\Ga,u) = \prod_{I \subset S} \det W_I(\rho,u)^{(-1)^{|I|+|S|}}.
 \end{equation}
(See Section 2 for more details.) One of the main {goals of} this paper is to generalize the above result to simple split algebraic groups of rank two. 

The paper is organized as follows. In Section 2, we review the definition of Poincar\'{e} series and their twistings and state Hashimoto's interpretation of Ihara's formula. Especially, the result is based on a length-preserving decomposition of the affine Weyl group of SL$_2$.
In Section 3, we study such decompositions for affine Weyl groups of rank two and discuss their relation with geometric zeta functions in Section 4. In Section 5, we state our conjecture about higher-dimensional cases and {discuss} the right-hand side of Eq.(\ref{Ih2}) for un-twisted Poincar\'{e} series.

\subsection*{Acknowledgements}
The authors would like to Professor Anton Deitmar and Professor Jiu-Kang Yu for their valuable discussions. The research was mainly performed while the first author was visiting Professor Deitmar in University of T\"{u}bingen. The authors would like to thank the university for its hospitality.

\section{Twisted Poincar\'{e} series and Ihara zeta functions} \label{sec-ranktwo}
\subsection{Poincar\'{e} series}

Let $(W,S)$ be a Coxeter system with a generating set $S$ consisting of elements of order two.
For $w \in W$, let $\ell(w)$ be the shortest length of a word consisting of elements of $S$ whose product is equal to $w$.
The Poincar\'{e} series associated to $(W,S)$ is a power series  with integer coefficients defined as
$$ W(u) = \sum_{w \in W}  u^{\ell(w)}.$$
For a subset $D$ of $W$, we also define $D(u) = \sum_{w \in D} u^{\ell(w)}.$
Especially, we are interested in the case where $D=W_I$, the subgroup generated by some subset $I \subset S$.
For instance, when $W$ is finite,
$$  \sum_{I \subset S} (-1)^{|I|}\frac{W(u)}{W_I(u)} = u^{\ell(w_0)}. $$
Here $w_0$ is the unique element of maximal length.
When $W$ is infinite,
$$  \sum_{I \subset S} (-1)^{|I|}\frac{W(u)}{W_I(u)} = 0,$$
which  implies that
\begin{equation} \label{formula1}
 W(u) = \left(\sum_{I \subsetneq S} \frac{(-1)^{|I|+|S|+1}}{W_I(u)} \right)\m.
 \end{equation}
See \cite{Hum}, Sections 1.11 and 5.12, for the proof of these statements.

\subsection{Hecke algebras}
For a Coxeter system $(W,S)$, a field $k$ and a formal parameter $q$  there is an associative $k$-algebra $H_q(W,S)$, called a Hecke algebra,  with generators $\{e_w\}_{w \in W}$. The multiplication of $H_q(W,S)$ is characterized by the following the relations:
\begin{eqnarray*}
 (e_s+1)(e_s-q)&=&0, \qquad \,\,\,\, \mbox{if }s \in S; \\
 e_w e_{v}&=& e_{wv}, \qquad \mbox{if }\ell(wv)= \ell(w)+\ell(v).
\end{eqnarray*}
Especially, if we set $q=1$, then $H_q(W,S)$ is isomorphic to the group algebra $k[W]$.

\subsection{Twisted Poincar\'{e} series}
For a representation $(\rho,V_\rho)$ of $H_{q}(W,S)$, consider the following power series
$$ D(\rho,u) = \sum_{w \in D} \rho(e_w) u^{\ell(w)} \in \mathrm{End}(V_\rho)[[u]],$$
called the Poincar\'{e} series of $D$ twisted by $\rho$. 
Note that when $D$ contains the identity element (e.g. $D=W_I$), the constant term of $D(\rho,u)$ is the identity operator and $D(\rho,u)$ has an inverse in $\mathrm{End}(V_\rho)[[u]]$.

Note that $e_s \mapsto q$ for all $s \in S$ induces a one-dimensional representation $\rho_1$ of $H_q(W,S)$. In this case,
$$ D(\rho_1,u) = \sum_{w \in D} \rho_1(e_w) u^{\ell(w)} =  \sum_{w \in D} q^{\ell(w)} u^{\ell(w)} = D(qu)$$
Therefore, one can regard the usual Poincar\'{e} series {as} a special case of twisted Poincar\'{e} series.

\subsection{Ihara zeta functions}
Let $W$ be the affine Weyl group of $G=$SL$_2$ over a local field $F$ with $p^{n}$ elements in its residue field. In this case, $S=\{s_1,s_2\}$ and $s_1 s_2$ has order infinity.
Let the formal parameter $q$ be equal to $p^n$, then the Hecke algebra $H_q(W,S)$ is isomorphic to the Iwahori-Hecke algebra of $G$.
Fix a discrete torsion-free cocompact subgroup $\Ga$ of $G$. Then the quotient of the Bruhat-Tits tree $\scrB$ of $G$ by $\Ga$ is a bipartite finite $(q+1)$-regular graph $X_\Ga$.  The Ihara zeta function of $\scrB_\Ga$ is defined as
$$ Z(\scrB_\Ga,u) = \prod_{c}(1-u^{\ell(c)})\m \in \Z[[u]].$$
Here $c$ runs over all equivalence classes of closed primitive geodesic walks in $\scrB_\Ga$ with length $\ell(c)$.

\bigskip

Let $\pi_\Ga$ be the representation of $H_q(W,S)$ induced from Iwahori-fixed vectors of $L^2(\Ga \backslash G)$.  In this case, elements in $H_q(W,S)$ can be regarded as operators on edges of $\scrB_\Ga$ and Hashimoto \cite{Ha1} showed that
\begin{theorem}
$Z(\scrB_\Ga,u) =\det(1- \pi_\Ga(e_{s_2 s_1}) u^2)^{\mathrm{-1}}.$
\end{theorem}

On the other hand, one can factor the group $W$ as a product of three subsets as follows.
\begin{equation}\label{Fatorization1}
W= \langle s_1 \rangle  \cdot \{ (s_2 s_1)^m\}_{m=0}^\infty \cdot \langle s_2 \rangle.
\end{equation}
Note that such a product always gives reduced words in $W$. Therefore, for any  representation $\rho$ of $H_q(W,S)$, we have
\begin{align*}
 W_{\{s_1,s_2\}}(\rho,u)  &= W_{\{s_1\}}(\rho,u) \bigg(\sum_{i=1}^\infty \rho(e_{s_2}e_{s_1})^m u^{2m} \bigg) W_{\{s_2\}}(\rho,u) \\
 &= W_{\{s_1\}}(\rho,u)  \left(I- \rho(e_{s_2}e_{s_1})u^2 \right)\m W_{\{s_2\}}(\rho,u)
\end{align*}
For  a finite dimensional representation $\rho$ of $H(W,G)$ over a field $K$, one can consider the determinant of $W_I(\rho,u)$, which is an invertible
element in the power series $K[[u]]$ (since its constant term is equal to one).
For convenience, define
$$ \Alt(W)(u)=  \prod_{I \subset S} W_I(u)^{(-1)^{|I|+|S|}}$$
and
$$\det \Alt(W)(\rho, u)=  \prod_{I \subset S} \det  W_I(\rho,u)^{(-1)^{|I|+|S|}}.$$

It was pointed out in \cite{Hof} that we can rewrite Hashimoto's result as the following.
\begin{theorem}
As a power series,
$$ Z(\scrB_\Ga,u) = \det \Alt(W)(\pi_\Ga,u).$$
\end{theorem}

\
{This of course only applies when the ground field $K$ has characteristic zero.} This interpretation suggests a possible way to generalize Ihara's theorem to higher rank cases.

\section{Affine Coxeter groups of rank two}

\subsection{Alternating products of Poincar\'{e} series} \label{sec-alt}
Suppse $(W,S=\{s_1,s_2,s_3\})$ is a Coxeter sytem for an affine Coxeter group of rank two.
Let $m_{ij}$ be the order of $s_i s_j$.  There are three types of such Coxeter systems up to isomorphism, characterized as the following:
\begin{itemize}
\item Type $\tilde{A_2}$: $(m_{12},m_{23},m_{13})=(3,3,3)$.
\item Type $\tilde{C_2}$: $(m_{12},m_{23},m_{13})=(4,2,4)$.
\item Type $\tilde{G_2}$: $(m_{12},m_{23},m_{13})=(6,2,3)$.
\end{itemize}
To evaluate $\Alt(W)(u)$, note that $W_I$ is a dihedral group or a cyclic group of order two when $I$ is a proper subset of $S$.
In this case, $W_I(u)$ can be written down directly. On the other hand, one can apply Eq. (\ref{formula1}) to compute the Poincar\'{e} series of $W(u)$ and obtain the following results.
$$
\Alt(W)(u)\m =
\begin{cases}
 (1-u^3)^2 &\mbox{, if $(W,S)$ is of type }\tilde{A}_2; \\
(1-u^4)(1-u^3) &\mbox{, if $(W,S)$ is of type }\tilde{C}_2; \\
(1-u^5)(1-u^3) &\mbox{, if $(W,S)$ is of type }\tilde{G}_2.\\
\end{cases}
$$
We shall show that the above identities can be extended to the case of twisted Poincar\'{e} series of Hecke algebras.

\subsection{Coxeter complexes}
Fix a geometric realization of $(W,S)$ on a real vector space $V$ of dimension two with an inner product $\langle, \rangle$ such that the corresponding metric is $W$-invariant.

The hyperplane $H_s$ fixed by an affine reflection $s$ in $W$ is called a wall. The set of walls  gives $V$ a simplicial structure and the resulting simplicial complex is called the Coxeter complex $\scrA$ of $W$.
Connected components of $V$ with all walls removed are open 2-simplices, called alcoves or chambers.

The unique chamber whose boundary is contained in $H_{s_1} \cup H_{s_2} \cup H_{s_3}$ is the fundamental chamber $\scrC$.
Label the facet (1-simplex) of $\scrC$ by $i$ if it is contained in $H_{s_i}$. Then one can extend this labelling to a $W$-invariant labelling on all facets in $\scrA$ uniquely.

Besides, there is a bijection between $W$ and chambers in $\scrA$ given by $w \mapsto w \scrC$. Moreover, if $w=s_{i_1} \cdots s_{i_n}$, then the chamber $w \scrC$ can be obtained by the gallery
$$C_0= \scrC \to C_1= s_{i_1} \scrC \to C_2= s_{i_1} s_{i_2} \scrC \to \cdots \to C_n= s_{i_1} \cdots s_{i_n} \scrC$$
such that the facet in $\bar{C}_k \cap \bar{C}_{k+1}$ is labelled by $i_{k+1}$. In other words, starting from the fundamental chamber $\scrC$, one can cross the facets labelled by $i_1,\cdots, i_n$ sequentially to arrive the chamber $w\scrC$.

\subsection{Straight strip}
Let $v_i = H_{s_{i+1}} \cap H_{s_{i+2}}$ be the vertex of the fundamental chamber $\scrC$ where the  subscripts are read modulo 3.
For convenience, we may assume that $v_3$ is the origin of $V$. Then the stabilizer of $v_3$ in $W$ is the linear subgroup $W_0$ generated by $s_1$ and $s_2$, which is the Weyl group. (This follows by our special choice of the ordering of $m_{ij}$ in the beginning of the section.) For the affine transformation $w \in W$, we can uniquely write $w = w_0w_t$ where $w_0 \in W_0$ is the linear part of $w$ and $w_t$ is a translation.

Now for $i=1,2$, removing all walls parallel to the vector $v_i$ from $V$, the connected component containing $\scrC$ in the resulting set is called the fundamental straight strip $\scrT_i$ in the direction $v_i$, which is the gray
region in Figure 1-4. For $w=w_0w_t \in W$, $w \scrT_i$ is a straight strip in the direction of $w_0 v_i$ and all such strips are called of type $i$.

\subsection{Stabilizer of $\scrT_i$} \label{stabilizer}
Let $\Stab(\scrT_i)$ be the stabilizer of $\scrT_i$ consisting of elements mapping $\scrT_i$ to itself and preserving the direction $v_i$.
Let $\scrL$ be the middle line of $\scrT_i$, then for $w \in \Stab(\scrT_i)$, $w$ must preserve $\scrL$ and its action on $\scrL$ has to be a translation (in the direction of $v_i$ or $-v_i$).
If the action of $w$ on $\scrL$ is trivial, $w$ fixes $\scrL$ point-wisely and it is the identity element or an affine reflection whose reflection axis is $\scrL$.
On the other hand, $\scrL$ is not a wall (by the construction of $\scrT_i$) and $w$ must therefore be the identity element in this case.
Therefore, elements in $\Stab(\scrT_i)$ are uniquely determined by their actions on $\scrL$.

Let $w_i$ be the element in  Stab$(\scrT_i)$ such that its action on $\scrL$ is the minimal translation in the direction of $v_i$. Then we have the following theorem.
\begin{theorem}
The stabilizer $Stab(\scrT_i)$ is a cyclic group generated by $w_i$.
\end{theorem}
\begin{minipage}{0.4\textwidth}
\begin{tikzpicture}[scale=1.8]
\clip (-0.4,-0.4) rectangle ++ (4.1,4.1);
\foreach \y in {-2,-1,...,4}
{
	\draw plot (\x+\y,1.732*\x);
	\draw plot (-\x+\y,1.732*\x);
	\draw plot (-\x,0.866*\y);
}
\filldraw [gray,opacity=.2]
(-5,0)--(10,0)--(10,.866)--(-5,.866);
\filldraw [gray,opacity=.2]
(0,0)--(1,0)--(0,-1.732)--(-1,-1.732);
\filldraw [gray,opacity=.2]
(5,8.66)--(6,8.66)--(1.5,.866)--(.5,.866);
\filldraw [gray,opacity=.5]
(1,0)--(0,0)--(.5,.866);
\filldraw [gray,opacity=.5]
(2,0)--(1.5,.866)--(2.5,.866);
\filldraw [gray,opacity=.5]
(1.5,.866)--(1,1.732)--(2,1.732);
\node at (.5,.3)  {$\scrC$};
\node at (.5,0){$1$};
\node at (.2,.5){$2$};
\node at (.8,.5){$3$};
\node at (3,.9){$\scrT_2$};
\node at (1.6,3){$\scrT_1$};
\node at (2,.5){$w_2 \scrC$};
\node at (1.5,1.4){$w_1 \scrC$};
\end{tikzpicture}

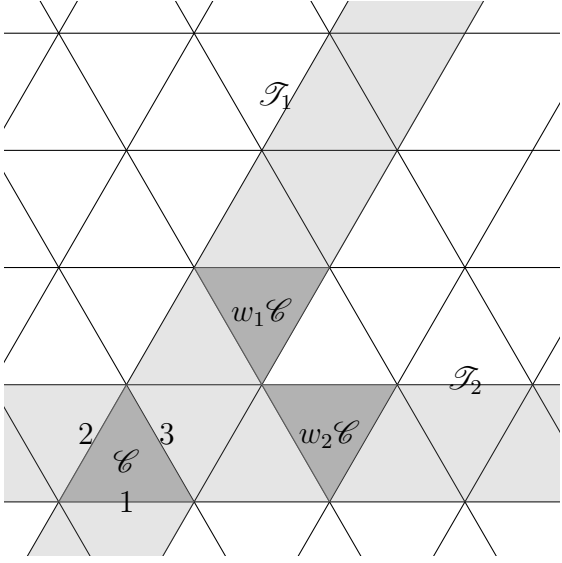
\captionof{figure}{Type $\tilde{A}_2$}
\end{minipage}
\hspace{1cm}
\begin{minipage}{0.4\textwidth}
\begin{tikzpicture}[scale=1.8]
\clip (-.5,-.5) rectangle ++ (4.2,4.2);
\filldraw [gray,opacity=.2]
(1,1)--(2,1)--(6,5)--(5,5);
\filldraw [gray,opacity=.2]
(0,0)--(1,0)--(-1,-2)--(-2,-2);
\filldraw [gray,opacity=.2]
(-2,0)--(5,0)--(5,1)--(-2,1);
\filldraw [gray,opacity=.5]
(1,0)--(0,0)--(.5,.5);
\filldraw [gray,opacity=.5]
(1.5,0.5)--(1,1)--(1.5,1.5)--(2,1);
\foreach \y in {-2,-1,...,5}
{
\draw plot (\x,\y);
\draw plot (\y,\x);
\draw plot (\x+\y,\x);
\draw plot (\x+\y,-\x);
}
\node at (.5,.2)  {{$\scrC$}};
\node at (.5,-.1) {$1$};
\node at (.2,.3) {$2$};
\node at (.8,.3) {$3$};
\node at (1.5, .8){$w_2 \scrC$};
\node at (1.5, 1.2){$w_1 \scrC$};
\node at (3.5, 1){$\scrT_2$};
\node at (3.2, 3.2){$\scrT_1$};
\end{tikzpicture}

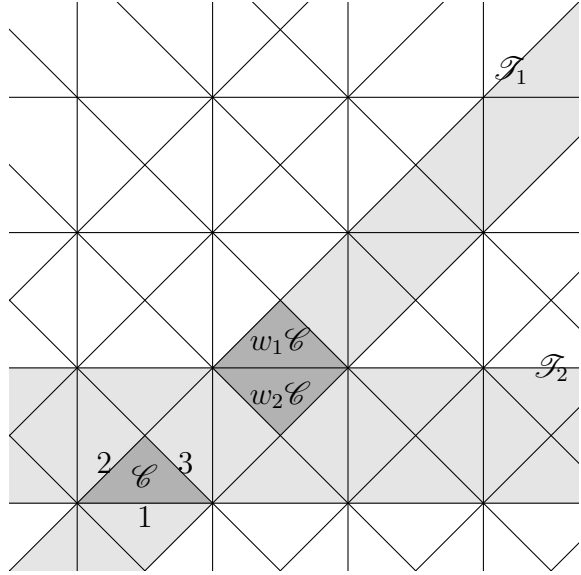
\captionof{figure}{Type $\tilde{C}_2$.}
\end{minipage}

\begin{minipage}{0.4\textwidth}
\begin{tikzpicture}[scale=2.4]
\clip (-.3,-2.1) rectangle ++ (3.2,3.2);
\filldraw [gray,opacity=.2]
(-1.5,-.866)--(7.5,4.33)--(7.5,2.578)--(-1.5,-2.578);
\filldraw [gray,opacity=.5]
(0,0)--(1,0)--(.75,.433);
\filldraw [gray,opacity=.7]
(0,0)--(1,0)--(.75,-.433);
\filldraw [gray,opacity=.5]
(1.5,-.866)--(2,0)--(2.25,-.433);
\filldraw [gray,opacity=.7]
(1.5,-.866)--(2,0)--(1.5,0);
\foreach \y in {-3,-2,...,4}
{
\draw plot (\x+\y,1.732*\x);
\draw plot (-\x+\y,1.732*\x);
\draw plot (-\x,0.866*\y);
\draw plot (1.5*\y, \x);
\draw plot(1.5*\x+3*\y, 0.866*\x);
\draw plot(-1.5*\x+3*\y, 0.866*\x);
}

\node  at (.7,.2)  {{$\scrC$}};
\node  at (.7,-.2) {$s_1 \scrC$};
\node  at (2,.-.4) {$w_1 \scrC$};
\node  at (1.6,.-.2) {$w_1 s_1 \scrC$};
\node  at (.5,0) {$1$};
\node  at (.9,.3) {$3$};
\node  at (.5,.3) {$2$};
\end{tikzpicture}
\captionof{figure}{$\scrT_1$ in type $\tilde{G}_2$.}
\end{minipage}
\hspace{2 cm}
\begin{minipage}{0.4\textwidth}
\begin{tikzpicture}[scale=2.4]
\clip (-.3,-1.1) rectangle ++ (3.2,3.2);
\filldraw [gray,opacity=.2]
(-1,0)--(6,0)--(6,.866)--(-1,.866);
\filldraw [gray,opacity=.5]
(0,0)--(1,0)--(.75,.433);
\filldraw [gray,opacity=.5]
(1.5,.866)--(2.5,.866)--(2.25,.433);
\foreach \y in {-3,-2,...,4}
{
\draw plot (\x+\y,1.732*\x);
\draw plot (-\x+\y,1.732*\x);
\draw plot (-\x,0.866*\y);
\draw plot (1.5*\y, \x);
\draw plot(1.5*\x+3*\y, 0.866*\x);
\draw plot(-1.5*\x+3*\y, 0.866*\x);
}

\node  at (.7,.2)  {{$\scrC$}};
\node  at (2.2,.7) {$w\scrC$};
\node  at (.5,0) {$1$};
\node  at (.9,.3) {$3$};
\node  at (.5,.3) {$2$};
\end{tikzpicture}
\captionof{figure}{$\scrT_2$ in type $\tilde{G}_2$.}
\end{minipage}

By direct computation, we obtain the following.
\begin{enumerate}
\item When $(W,S)$ is of type $\tilde{A}_2$, we have  $w_1=s_3 s_2 s_1$ and $w_2 =  s_3 s_1 s_2$.
\item When $(W,S)$ is of type $\tilde{C}_2$, we have  $w_1= s_3 s_1 s_2 s_1$ and $w_2= s_3 s_1 s_2$.
\item When $(W,S)$ is of type $\tilde{G}_2$, we have $w_1 = s_3 s_1 s_2 s_3 s_1  $ and
$w_2 = s_3 s_1 s_2 s_1 s_2$.
\end{enumerate}
One can check case by case to see that except for $w_1$ in the case of type $\tilde{G_2}$,
\begin{itemize}
\item[(*)] the length of $(w_i)^k$ is equal to $|k|$ times of the length of $w_i$.
\end{itemize}
Now, in the case of $\tilde{G_2}$ we shall replace $w_1$ by
$$ w_1' = s_1 w_1 s_1 = (s_1 s_3 s_1) s_2 s_3  = (s_3 s_1 s_3 )s_2  s_3 = s_3 s_1 (s_3 s_2  s_3) = s_3 s_1 s_2.$$
which is a generator of the stabilizer of $s_1 \scrT_1$. Then $w_1'$ will also satisfy the property $(*)$.
To abuse the notation, we will also denote $w_1'$ by $w_1$ (in the case of the group $\tilde{G_2}$) in the rest of paper.

Let $H_i = \{ (w_i)^k \}_{k \in \Z_\geq 0}.$ Note that
$$ H_i(u)= \sum_{i=0}^\infty u^{\ell(w_i^k)}  = \sum_{i=0}^\infty  u^{k \ell(w_i)} = (1- u^{\ell(w_i)})\m.$$
Therefore, we can rewrite the result in \S \ref{sec-alt} as
\begin{proposition} \label{prop1}
For any affine Coxeter group $(W,S)$ of rank two, the following hold.
$$ H_1(u)H_2(u)=\Alt(W)(u) \qquad \mbox{and} \qquad H_1(u)H_2(u)=\Alt(W)(u).$$
\end{proposition}
Next, we will show that the above theorem can be extended to twisted Poincar\'{e} series on Hecke algebras.
Note that the twisted Poincar\'{e} series does not commute so the ordering in the alternating product does matter.
We will prove the following theorem in the rest of the section.
\begin{theorem} \label{maintheorem1}
For any representation $\rho$ of $H_q(W,S)$,
there exists certain ordering of
$$\bigg\{ H_1(\rho,u), H_2(\rho,u), W_{I}(\rho,u)^{(-1)^{|I|}}, I \subsetneq S\bigg\}$$
such that their product under this ordering is equal to $W(\rho,u)$.
\end{theorem}
Immediately, we have the following corollary.
\begin{corollary} \label{corollary1}
When $\rho$ is finite dimensional,
$$ \det H_1(\rho,u)  \det H_2(\rho,u) = \det \Alt(W)(\rho,u).$$
\end{corollary}

\subsection{Factorization of Coxeter groups}
To prove Theorem \ref{maintheorem1}, we need the following lemma.
\begin{lemma} \label{lemma1}
Let $D_1,\cdots, D_m$ be non-empty subsets of a Coxeter group $W$. If $W=D_1\times \cdots \times D_m$ and $D_1(u)\cdots D_m(u)=W(u)$, then
this factorization is length-preserving such that
for $w_i \in D_i$, $\ell(w_1 \cdots w_m) = \ell(w_1)+ \cdots +\ell(w_m)$.
\end{lemma}
\begin{proof}
Since $\ell(w_1 \cdots w_m)  \leq \ell(w_1)+ \cdots +\ell(w_m)$, for any integer $k$ we have
\begin{eqnarray*}
\Omega_k &:=& \{ (w_1,\cdots,w_m)  \in D_1 \times \cdots \times D_m, \ell(w_1)+ \cdots +\ell(w_m)\leq k\} \\
& \subseteq &\{ (w_1,\cdots,w_m)  \in D_1 \times \cdots \times D_m, \ell(w_1\cdots w_m)\leq k\} \\
& = &\{ w \in W, \ell(w) \leq k\}=: \Omega_k'.
\end{eqnarray*}
On the other hand,
\begin{align*}
D_1(u)\cdots D_m(u) = \prod_{i=1}^m \left(\sum_{w_i \in D_i} u^{\ell(w_i)}\right) =|\Omega_{0}|+ \sum_{i=1}^\infty \left(|\Omega_{i}|-|\Omega_{i-1}|\right)u^i
\end{align*}
and
\begin{align*}
W(u) =  |\Omega_{0}'| + \sum_{i=1}^\infty \left(|\Omega_{i}'|-|\Omega_{i-1}'|\right)u^i
\end{align*}
From $D_1(u)\cdots D_m(u)=W(u)$, we conclude that $|\Omega_{i}|= |\Omega_{i}'|$ for all $i$. Therefore, $\Omega_k = \Omega_k'$ and $\ell(w_1 \cdots w_m) = \ell(w_1)+ \cdots +\ell(w_m)$ for $w_i \in D_i$.
\end{proof}

Note that when  $\ell(w_1 \cdots w_m) = \ell(w_1)+ \cdots +\ell(w_m)$, for any representation $\rho$ of $H_q(W,S)$, we have
$$ \rho(e_{w_1 \cdots w_m}) = \rho(e_{w_1}) \cdots \rho(e_{w_m}).$$
Together with the above lemma, we have
\begin{theorem} \label{theorem2}
If $W=D_1\times \cdots \times D_m$ and $D_1(u)\cdots D_m(u)=W(u)$, then for any representation $\rho$ of $H_q(W,S)$
$$ W(\rho,u) = D_1(\rho,u) \cdots D_m(\rho,u).$$
\end{theorem}

Let us give some examples of the above theorem. For subsets $I \subset J \subset S$, define
$$ W_{J/I} = \{ w \in W_J,  \ell (ws)>  \ell(w) , \forall s \in I\} \qquad \mbox{and} \qquad W_{I\backslash J} = \{ w \in W,  \ell (sw)>  \ell(w) , \forall s \in I\}
$$
which are the sets of minimal left $W_I$-coset representatives and minimal right $W_I$-coset representatives of $W_J$ respectively.
\begin{theorem}[\cite{Hum}, \S 1.11]
For subsets $I \subset J \subset S$, the following hold.
\begin{enumerate}
\item $W_J = W_{J/I} \times W_I = W_I \times  W_{I\backslash J}.$
\item $W_J(u) = W_{J/I}(u) W_I(u) = W_I(u) W_{I\backslash J}(u).$
\end{enumerate}
\end{theorem}
For convenience, we also denote the set of generators $I=\{s_{i_1},\cdots,s_{i_k}\}$ by its set of indices $\{i_1,\cdots,i_k\}$.
Applying Theorem \ref{theorem2}, to prove Theorem \ref{maintheorem1}, it is sufficient to prove the following theorem.
\begin{theorem} \label{theorem3}
Let $D_1,\cdots, D_m$ be subsets of $W$ given by the following.
\begin{itemize}
\item $(D_1,\cdots,D_5) = (W_{\{1,2\}/\{2\}},H_1,W_{\{2,3\}/\{3\}}, H_2, W_{\{1\}\backslash\{1,3\}})$ if $(W,S)$ is of type $\tilde{A}_2$ or $\tilde{C}_2$.
\item $(D_1,\cdots,D_4) = (W_{\{2,1\}/\{1\}},H_1, H_2, W_{\{1,3\}})$ if $(W,S)$ is of type $\tilde{G}_2$.
\end{itemize}
Then $D_1(u)\cdots D_m(u)=W(u)$ and $W=D_1 \times \cdots \times D_m$.
\end{theorem}

Note that for type $\tilde{A}_2$ or $\tilde{C}_2$, by Proposition \ref{prop1},
\begin{align*}
D_1(u)\cdots D_5(u)&= W_{\{1,2\}/\{2\}}(u) H_1 (u) W_{\{2,3\}/\{3\}} (u) H_2(u) W_{\{1\}\backslash\{1,3\}}(u)\\
&= \frac{ W_{\{1,2\}}(u)W_{\{2,3\}}(u)W_{\{1,3\}}(u)}{W_{\{1\}}(u)W_{\{2\}}(u)W_{\{3\}}(u)} \Alt(W)(u) = W(u).
\end{align*}
For type $\tilde{G}_2$, note that $W_{\{2,3\}}=\{e, s_2, s_3, s_2 s_3\} = W_{\{2\}} \times W_{\{3\}}.$ Therefore,
$$W_{\{2,3\}}(u)=W_{\{2\}}(u)W_{\{3\}}(u)$$
 and
\begin{align*}
D_1(u)\cdots D_4(u)&= W_{\{2,1\}/\{1\}}(u) H_1 (u) H_2(u) W_{\{1,3\}}(u)\\
&= \frac{ W_{\{2,1\}}(u)W_{\{1,3\}}(u)}{W_{\{1\}}(u)} \Alt(W)(u)\cdot \frac{W_{\{2,3\}}(u)}{W_{\{2\}}(u)W_{\{3\}}(u)} = W(u).
\end{align*}

It remains to show that  $W=D_1 \times \cdots \times D_m$. Our strategy is to study the geometric interpretation of this factorization.
As we mentioned before, for each $w=s_{i_1}\cdots s_{i_k} \in W$,
one can start from the fundamental chamber $C$ and cross the facets labelled by $i_1,\cdots, i_n$ sequentially to arrive the chamber $wC$.
Therefore, one can find the collections of chambers  $\mathfrak{C}_{m+1}=\{C\}$, $\mathfrak{C}_m= D_m( \mathfrak{C}_{m+1})$, $\cdots, \mathfrak{C}_1 =D_1(\mathfrak{C}_{2})$ step by step.
Instead of writing down the details, we just draw the figure for $\mathfrak{C}_i$ for each type in the following.
At each stage, the whole gray area is the region $\mathfrak{C}_i$.
connected components with red boundary are copies of $\mathfrak{C}_{i+1}$.
Moreover, the set of dark gray triangles is the image of $D_{i}(C)$.
From these figures, we see that $\mathfrak{C}_1=D_1 \cdots D_m(C)$ and $W(C)$ are bijective. We conclude that $W=D_1 \times \cdots \times D_m$.

\subsection*{Type $\tilde{A}_2$}
$ $\\
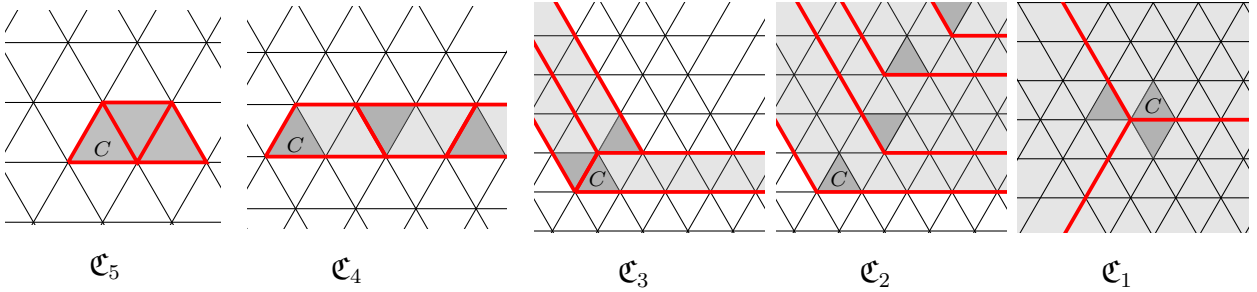
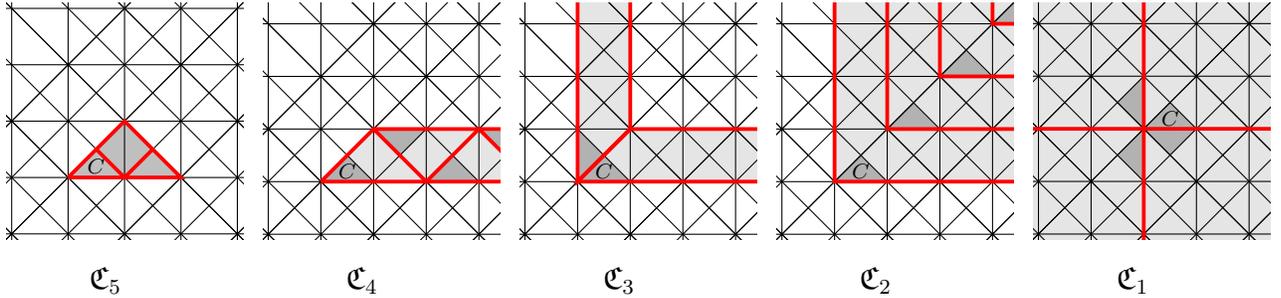
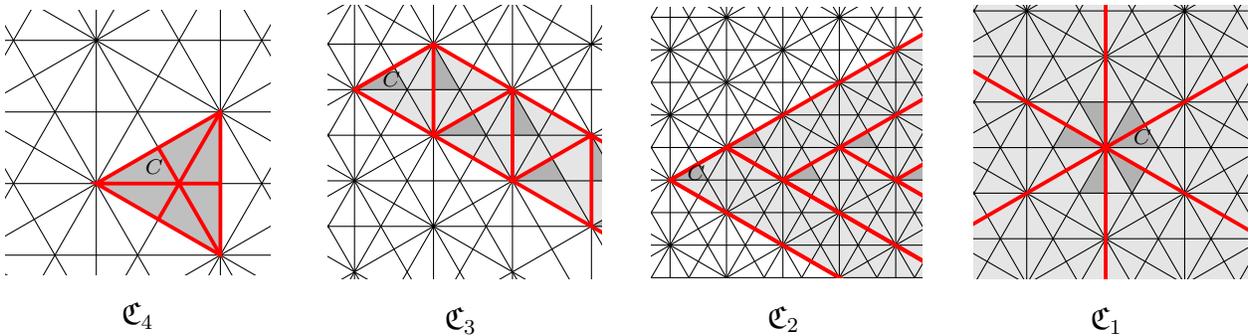
\begin{center}
\begin{minipage}{0.15\textwidth}
\begin{tikzpicture}[scale=.92]
\clip (-0.9,-0.9) rectangle ++ (3.1,3.1);
\foreach \y in {-2,-1,...,4}
{
	\draw plot (\x+\y,1.732*\x);
	\draw plot (-\x+\y,1.732*\x);
	\draw plot (-\x,0.866*\y);

}
\filldraw [gray,opacity=.5]
(0,0)--(2,0)--(1.5,.866)--(.5,.866);
\node at (.5,.2) {\tiny{$C$}};
\draw [line width=.5mm, red]  (0,0)--(2,0)--(1.5,.866)--(.5,.866)--(0,0);
\draw [line width=.5mm, red]  (1.5,.866)--(1,0)--(.5,.866);
\end{tikzpicture}
$$ \mathfrak{C}_5$$
\end{minipage}
\hspace{.3cm}
\begin{minipage}{0.15\textwidth}
\begin{tikzpicture}[scale=.8]
\clip (-0.3,-1.2) rectangle ++ (4.3,3.7);
\foreach \y in {-2,-1,...,4}
{
	\draw plot (\x+\y,1.732*\x);
	\draw plot (-\x+\y,1.732*\x);
	\draw plot (-\x,0.866*\y);
}

\filldraw [gray,opacity=.2]
(0,0)--(10,0)--(10.5,.866)--(.5,.866);
\filldraw [gray,opacity=.5]
(1,0)--(0,0)--(.5,.866);
\filldraw [gray,opacity=.5]
(2,0)--(1.5,.866)--(2.5,.866);
\filldraw [gray,opacity=.5]
(4,0)--(3,0)--(3.5,.866);
\node at (.5,.2) {\tiny{$C$}};
\draw [line width=.5mm, red]  (0,0)--(2,0)--(1.5,.866)--(.5,.866)--(0,0);
\draw [line width=.5mm, red]  (3,0)--(2,0)--(1.5,.866)--(3.5,.866)--(3,0);
\draw [line width=.5mm, red]  (3,0)--(5,0)--(4.5,.866)--(3.5,.866)--(3,0);
\end{tikzpicture}
$$ \mathfrak{C}_4$$
\end{minipage}
\hspace{.9cm}
\begin{minipage}{0.15\textwidth}
\begin{tikzpicture}[scale=.6]
\clip (-0.9,-0.9) rectangle ++ (5.1,5.1);
\foreach \y in {-2,-1,...,4}
{
	\draw plot (\x+\y,1.732*\x);
	\draw plot (-\x+\y,1.732*\x);
	\draw plot (-\x,0.866*\y);
}

\filldraw [gray,opacity=.2]
(0,0)--(10,0)--(10.5,.866)--(-.5,.866);
\filldraw [gray,opacity=.2]
(-.5,.866)--(-5,8.66)--(-3,8.66)--(1.5,.866);
\filldraw [gray,opacity=.5]
(1,0)--(0,0)--(.5,.866);
\filldraw [gray,opacity=.5]
(0,0)--(.5,.866)--(-.5,.866);
\filldraw [gray,opacity=.5]
(1,1.732)--(.5,.866)--(1.5,.866);
\node at (.5,.3) {\tiny{$C$}};
\draw [line width=.5mm, red]  (0,0)--(12,0)--(11.5,.866)--(.5,.866)--(0,0);
\draw [line width=.5mm, red]  (0,0)--(-5,8.66)--(-4,8.66)--(.5,.866)--(0,0);
\draw [line width=.5mm, red]  (1.5,.866)--(-3,8.66);
\end{tikzpicture}
$$ \mathfrak{C}_3$$
\end{minipage}
\hspace{.3cm}
\begin{minipage}{0.15\textwidth}
\begin{tikzpicture}[scale=.6]
\clip (-0.9,-0.9) rectangle ++ (5.1,5.1);
\filldraw [gray,opacity=.5]
(1, 0)--(0,0)--(0.5,.866);
\filldraw [gray,opacity=.5]
(2.5, 2.598)--(1.5,2.598)--(2,.866*4);
\filldraw [gray,opacity=.5]
(2.5, 4.33)--(3.5,4.33)--(3,.866*4);
\filldraw [gray,opacity=.5]
(1, 1.732)--(2,1.732)--(1.5,.866);

\foreach \y in {-2,-1,...,7}
{
	\draw plot (\x+\y,1.732*\x);
	\draw plot (-\x+\y,1.732*\x);
	\draw plot (-\x,0.866*\y);
}

\filldraw [gray,opacity=.2]
(0,0)--(10,0)--(10,8.66)--(-2.5,4.33);
\node at (.5,.3) {\tiny{$C$}};
\draw [line width=.5mm, red]  (-1,1.732)--(0,0)--(6,0);
\draw [line width=.5mm, red]  (-1,5.196)--(1.5,.866)--(6,.866);
\draw [line width=.5mm, red]  (-1,6.898)--(1.5,2.598)--(6,2.598);
\draw [line width=.5mm, red]  (2,.866*6)--(3,.866*4)--(6,.866*4);
\end{tikzpicture}
$$ \mathfrak{C}_2$$
\end{minipage}
\hspace{.3cm}
\begin{minipage}{0.15\textwidth}
\begin{tikzpicture}[scale=.6]
\clip (-2.5,-2.5) rectangle ++ (5.1,5.1);
\filldraw [gray,opacity=.2]
(-6,-6)--(-6,6)--(6,6)--(6,-6);

\foreach \y in {-2,-1,...,7}
{
	\draw plot (\x+\y,1.732*\x);
	\draw plot (-\x+\y,1.732*\x);
	\draw plot (-\x,0.866*\y);
}

\filldraw [gray,opacity=.5]
(1,0)--(0,0)--(.5,.866);
\filldraw [gray,opacity=.5]
(1,0)--(0,0)--(.5,-.866);
\filldraw [gray,opacity=.5]
(-1,0)--(0,0)--(-.5,.866);
\node at (.5,.3) {\tiny{$C$}};
\draw [line width=.5mm, red]  (-5,8.66)--(0,0)--(6,0);
\draw [line width=.5mm, red]  (0,0)--(-5,-8.66);
\end{tikzpicture}
$$ \mathfrak{C}_1$$
\end{minipage}
\end{center}

\subsection*{Type $\tilde{C}_2$}
$ $\\

\begin{center}
\begin{minipage}{0.15\textwidth}
\begin{tikzpicture}[scale=.75]
\clip (-1.1,-1.1) rectangle ++ (4.2,4.2);
\filldraw [gray,opacity=.5]
(2,0)--(0,0)--(1,1);
\foreach \y in {-2,-1,...,5}
{
\draw plot (\x,\y);
\draw plot (\y,\x);
\draw plot (\x+\y,\x);
\draw plot (\x+\y,-\x);
}
\foreach \y in {-2,0,...,4}
{
\draw plot   (\x+\y,\x);
\draw plot   (\x+\y,-\x);
}
\node at (.5,.2)  {\tiny{$C$}};
\draw [line width=.5mm,red] (0,0)--(1,1)--(2,0)--(0,0);
\draw [line width=.5mm,red] (.5,.5)--(1,0)--(1.5,.5);
\end{tikzpicture}
$$ \mathfrak{C}_5$$
\end{minipage}
\hspace{.5cm}
\begin{minipage}{0.15\textwidth}
\begin{tikzpicture}[scale=.7,domain=-10:10 ]
\clip (-1.1,-1.1) rectangle ++ (4.5,4.5);
\filldraw [gray,opacity=.2]
(0,0)--(1,1)--(10,1)--(10,0);
\filldraw [gray,opacity=.5]
(1,0)--(0,0)--(.5,.5);
\filldraw [gray,opacity=.5]
(3,0)--(2,0)--(2.5,.5);
\filldraw [gray,opacity=.5]
(5,0)--(4,0)--(4.5,.5);
\filldraw [gray,opacity=.5]
(1,1)--(2,1)--(1.5,.5);
\filldraw [gray,opacity=.5]
(3,1)--(4,1)--(3.5,.5);
\foreach \y in {-2,-1,...,7}
{
\draw plot (\x,\y);
\draw plot (\y,\x);
\draw plot (\x+\y,\x);
\draw plot (\x+\y,-\x);
}
\foreach \y in {-2,0,...,7}
{
\draw plot   (\x+\y,\x);
\draw plot   (\x+\y,-\x);
}
\node at (.5,.2)  {\tiny{$C$}};
\draw [line width=.5mm,red] (0,0)--(1,1)--(2,0)--(3,1)--(4,0)--(5,1);
\draw [line width=.5mm,red] (0,0)--(10,0);
\draw [line width=.5mm,red] (1,1)--(10,1);
\end{tikzpicture}
$$ \mathfrak{C}_4$$
\end{minipage}
\hspace{.5cm}
\begin{minipage}{0.15\textwidth}
\begin{tikzpicture}[scale=.7,domain=-10:10 ]
\clip (-1.1,-1.1) rectangle ++ (4.5,4.5);
\filldraw [gray,opacity=.2]
(0,0)--(1,1)--(10,1)--(10,0);
\filldraw [gray,opacity=.2]
(0,0)--(1,1)--(1,10)--(0,10);
\filldraw [gray,opacity=.5]
(1,0)--(0,0)--(.5,.5);
\filldraw [gray,opacity=.5]
(0,0)--(0,1)--(.5,.5);
\foreach \y in {-2,-1,...,7}
{
\draw plot (\x,\y);
\draw plot (\y,\x);
\draw plot (\x+\y,\x);
\draw plot (\x+\y,-\x);
}
\foreach \y in {-2,0,...,7}
{
\draw plot   (\x+\y,\x);
\draw plot   (\x+\y,-\x);
}
\node at (.5,.2)  {\tiny{$C$}};
\draw [line width=.5mm,red] (10,1)--(1,1)--(0,0)--(10,0);
\draw [line width=.5mm,red] (1,10)--(1,1)--(0,0)--(0,10);
\end{tikzpicture}
$$ \mathfrak{C}_3$$
\end{minipage}
\hspace{.5cm}
\begin{minipage}{0.15\textwidth}
\begin{tikzpicture}[scale=.7,domain=-10:10 ]
\clip (-1.1,-1.1) rectangle ++ (4.5,4.5);
\filldraw [gray,opacity=.2]
(0,0)--(0,10)--(10,10)--(10,0);

\foreach \y in {0,1,...,5}
{
	\filldraw [gray,opacity=.5]
	(1+\y,\y)--(\y,\y)--(.5+\y,.5+\y);
}

\foreach \y in {-2,-1,...,7}
{
\draw plot (\x,\y);
\draw plot (\y,\x);
\draw plot (\x+\y,\x);
\draw plot (\x+\y,-\x);
}
\foreach \y in {-2,0,...,7}
{
\draw plot   (\x+\y,\x);
\draw plot   (\x+\y,-\x);
}
\node at (.5,.2)  {\tiny{$C$}};
\foreach \y in {0,1,...,5}
{
  \draw [line width=.5mm,red] (\y,5+\y)--(\y,\y)--(5+\y,\y);	
}
\end{tikzpicture}
$$ \mathfrak{C}_2$$
\end{minipage}
\hspace{.5cm}
\begin{minipage}{0.15\textwidth}
\begin{tikzpicture}[scale=.7,domain=-10:10 ]
\clip (-2.1,-2.1) rectangle ++ (4.5,4.5);
\filldraw [gray,opacity=.2]
(-10,-10)--(-10,10)--(10,10)--(10,-10);

\filldraw [gray,opacity=.5]
(0,0)--(.5,.5)--(1,0)--(.5,-.5);

\filldraw [gray,opacity=.5]
(0,0)--(0,1)--(-.5,.5);
\filldraw [gray,opacity=.5]
(0,0)--(0,-1)--(-.5,-.5);
\foreach \y in {-3,-2,...,4}
{
\draw plot (\x,\y);
\draw plot (\y,\x);
\draw plot (\x+\y,\x);
\draw plot (\x+\y,-\x);
}
\foreach \y in {-4,-2,...,4}
{
\draw plot   (\x+\y,\x);
\draw plot   (\x+\y,-\x);
}
\node at (.5,.2)  {\tiny{$C$}};

\draw [line width=.5mm,red] (-10,0)--(10,0);
\draw [line width=.5mm,red] (0,-10)--(0,10);		
\end{tikzpicture}
$$ \mathfrak{C}_1$$
\end{minipage}
\end{center}

\subsection*{Type $\tilde{G}_2$}
$ $\\
\begin{center}
\begin{minipage}{0.2\textwidth}
\begin{tikzpicture}[scale=1.1]
\clip (-1.1,-1.1) rectangle ++ (3.2,3.2);
\filldraw [gray,opacity=.5]
(0,0)--(1.5,.866)--(1.5,-.866);
\foreach \y in {-3,-2,...,4}
{
\draw plot (\x+\y,1.732*\x);
\draw plot (-\x+\y,1.732*\x);
\draw plot (-\x,0.866*\y);
\draw plot (1.5*\y, \x);
\draw plot(1.5*\x+3*\y, 0.866*\x);
\draw plot(-1.5*\x+3*\y, 0.866*\x);
}

\draw [line width=.5mm, red]
(0,0)--(1.5,.866)--(1.5,-.866)--(0,0)--(1.5,0);
\draw [line width=.5mm, red]
(.75,.433)--(1.5,-.866);
\draw [line width=.5mm, red]
(.75,-.433)--(1.5,.866);
\node  at (.7,.2)  {\tiny{$C$}};
\end{tikzpicture}
$$ \mathfrak{C}_4$$
\end{minipage}
\hspace{.5cm}
\begin{minipage}{0.2\textwidth}
\begin{tikzpicture}[scale=.7]
\coordinate (A) at (1.5,.866);

\clip (-.5,-3.6) rectangle ++ (5.2,5.2);
\filldraw [gray, opacity=.2]
(0,0)--(1.5,.866)--(18,-8.66)--(15,-8.66);
\foreach \y in {0,1,...,4}
{
	\filldraw [gray,opacity=.5]
   (1.5*\y,-.866*\y)--(1+1.5*\y,-.866*\y)--(.75+1.5*\y,.433-.866*\y);
   \filldraw [gray,opacity=.5]
   (1.5*\y+1.5,-.866*\y+.866)--(1.5+1.5*\y,-.866*\y)--(2+1.5*\y,-.866*\y);
}

\foreach \y in {-3,-2,...,4}
{
\draw plot (\x+\y,1.732*\x);
\draw plot (-\x+\y,1.732*\x);
\draw plot (-\x,0.866*\y);
\draw plot (1.5*\y, \x);
\draw plot(1.5*\x+3*\y, 0.866*\x);
\draw plot(-1.5*\x+3*\y, 0.866*\x);
}

\draw [line width=.5mm, red]  (0,0)--(1.5,0.866)--(1.5,-.866)--(3,0)--(3,-1.732)--(4.5,-.866)--(4.5,-2.596);
\draw [line width=.5mm, red]
(9,-3.472)--(1.5,.866)--(0,0)--(7.5,-4.33);
\node  at (.7,.2)  {\tiny{$C$}};
\end{tikzpicture}
$$ \mathfrak{C}_3$$
\end{minipage}
\hspace{.5cm}
\begin{minipage}{0.2\textwidth}
\begin{tikzpicture}[scale=.5, domain=-8:10]
\clip (-.5,-2.6) rectangle ++ (7.2,7.2);
\filldraw [gray, opacity=.2]
(15,8.66)--(0,0)--(15,-8.66);
\foreach \y in {0,1,...,4}
{
	\filldraw [gray,opacity=.5]
   (3*\y,0)--(1+3*\y,0)--(.75+3*\y,.433);
   \filldraw [gray,opacity=.5]
    (1.5+3*\y,.866)--(2.5+3*\y,.866)--(2.25+3*\y,1.299);
 }

\foreach \y in {-5,-4,...,9}
{
\draw plot (\x+\y,1.732*\x);
\draw plot (-\x+\y,1.732*\x);
\draw plot (-\x,0.866*\y);
\draw plot (1.5*\y, \x);
\draw plot(1.5*\x+3*\y, 0.866*\x);
\draw plot(-1.5*\x+3*\y, 0.866*\x);
}

\foreach \y in {0,1,...,4}
{
    \draw [line width=.5mm, red]
	(3*\y+9,-3.472)--(3*\y+1.5,.866)--(3*\y,0)--(3*\y+7.5,-4.33);
	\draw [line width=.5mm, red]
	(1.5+3*\y,.866)--(3*\y+9,5.193);
 }

\draw [line width=.5mm, red]
(9,-3.472)--(1.5,.866)--(0,0)--(7.5,-4.33);
\node  at (.7,.2)  {\tiny{$C$}};
\end{tikzpicture}
$$ \mathfrak{C}_2$$
\end{minipage}
\hspace{.5cm}
\begin{minipage}{0.2\textwidth}
\begin{tikzpicture}[scale=.7, domain=-8:10]
\clip (-2.5,-2.5) rectangle ++ (5.2,5.2);
\filldraw [gray, opacity=.2]
(5,-5)--(5,5)--(-5,5)--(-5,-5);

\filldraw [gray,opacity=.5]
(0,0)--(1,0)--(.5,.866);

\filldraw [gray,opacity=.5]
(0,0)--(.5,-.866)--(.75,-.433);

\filldraw [gray,opacity=.5]
(0,0)--(0,.866)--(-.5,.866);

\filldraw [gray,opacity=.5]
(0,0)--(-1,0)--(-.75,.433);

\filldraw [gray,opacity=.5]
(0,0)--(0,-.866)--(-.5,-.866);

\foreach \y in {-5,-4,...,9}
{
\draw plot (\x+\y,1.732*\x);
\draw plot (-\x+\y,1.732*\x);
\draw plot (-\x,0.866*\y);
\draw plot (1.5*\y, \x);
\draw plot(1.5*\x+3*\y, 0.866*\x);
\draw plot(-1.5*\x+3*\y, 0.866*\x);
}

\draw [line width=.5mm, red]
(7.5,4.33)--(-7.5,-4.33);
\draw [line width=.5mm, red]
(-7.5,4.33)--(7.5,-4.33);
\draw [line width=.5mm, red]
(0,5)--(0,-5);
\node  at (.7,.2)  {\tiny{$C$}};
\end{tikzpicture}
$$ \mathfrak{C}_1$$
\end{minipage}
\end{center}

\section{Zeta functions of straight strips}

\subsection{Finite quotients of buildings}

Let $G$ be a simply-connected connected split simple algebraic group of rank two over a local field $F$ with $q$ elements in its residue field.
Fix a maximal split torus $T$ with the maximal compact subgroup $T_0\leq T$ and the normalizer $N$.
Let $W=N/T_0$ be the affine Weyl group of $G$ with a set of generators $S$ corresponding to a fixed choice of positive roots $\Phi^+$. Let $\I$ be the corresponding Iwahori subgroup.

The Coxeter system $(W,S)$ is of type $\tilde{A}_2$, $\tilde{C}_2$, or $\tilde{G}_2$, when
$G=$ SL$_3$, Spin$_4$, or G$_2$ respectively.

Let $\scrB$ be the Bruhat-Tits building of $G$ whose chambers are parametrized by the cosets $G/\I$.
We shall identify the Coxeter complex $\scrA$ of the Coxeter system $(W,S)$ with the standard apartment corresponding to the torus $T$ and represent the standard chamber $\scrC$ in $\scrA$ by the Iwahori subgroup $\I$.


Fix a discrete torsion-free cocompact subgroup $\Gamma$ of $G$, the quotient $\scrB_\Gamma$ is a finite complex with the fundamental group $\Gamma$.

\subsection{Stabilizer of straight strips}
Let $\scrT:=\scrT_i$ be the standard straight strip in the direction of $v=v_i$ which contains the standard chamber $\scrC$ when $(G,i) \neq (G_2,1)$. When $i=1$ and $G=G_2$, we set $\scrT= s_1 \scrT_i$ and $v= s_1 v_1$. (See \S \ref{stabilizer} for the discussion of this special case.) Let $w:=w_i$. 

Note that for a simplicial automorphism $\sigma$ on $\scrT$, its action on the middle line $\scrL$ is either an affine reflection (which has a fixed point) or a translation by $k v$ for some real number $k$.
In the later case, when $k>0$, $\sigma\big|_\scrL$ is called a positive translation.

Now let $\Aut(\scrT)$ be the group of simplicial automorphisms of $\scrT$ and consider the following subgroup of $\Aut(\scrT)$.
\begin{align*}
\Aut_1(\scrT) &= \{ \sigma \in \Aut(\scr\scrT), \sigma|_\scrL \mbox{ is a translation}\} \\
\Aut_2(\scrT) &= \{ \sigma \in \Aut(\scrT), \sigma \mbox{ is type-preserving}\}.
\end{align*}
Then one can verify that $\Aut_1(\scrT)\cap \Aut_2(\scrT)$ is a cyclic group generated by $w=w_i$, on a case-by-case basis from Figures 1--4 in Section 2.

Since $G$ is simply-connected, the action of $G$ on the building is type-preserving. Moreover, the assumption that $\Gamma$ is discrete and torsion-free implies that no non-identity element can have a fixed point in the building. Thus, the setwise stabilizer of $\scrT$ in $\Gamma$ acts faithfully on $\scrT_i$ and
$$\langle \ga_0 \rangle = \Stab_\Ga(\scrT_i) \cong \Stab_\Ga(\scrT) \big|_{\scrT} \leq \langle w \rangle $$
for some unique $\ga_0$ with $\ga_0\big|_{\scrL}$ being a positive translation.
A similar result is true of any other straight strip in the building $\scrB$. Thus, immediately, we have the following lemma.
\begin{lemma} \label{lemma31}
For $\ga \in \Ga$, if $\ga(g\scrT)=g\scrT$, then $\ga g\scrC = g w^{k} \scrC$ for some $k \in \Z$. Especially, it implies that $g^{-1} \ga g\in I w^k I$.
\end{lemma}

\subsection{Closed pointed straight strips}

Consider the following set
\begin{align*} \scrP &=\{(\ga, g\scrC,g\scrT), g \in G, \gamma \in \Gamma, \gamma(g\scrT) = g\scrT, \gamma\mid_{g \scrL} \mbox{is a positive translation}\}\\
&=\{(\ga, g\scrC,g\scrT), g \in G, \gamma \in \Gamma, \gamma(g\scrT) = g\scrT, g^{-1}\gamma g \mid_{\scrL} \mbox{is a positive translation}\}.
\end{align*}

endowed with an equivalence relation defined as follows: $(\ga_1, g_1\scrC, g_1\scrT ) \sim (\ga_2, g_2\scrC, g_2\scrT)$ if there exists some $\gamma \in \Gamma$, such that $$ \gamma g_1 \scrT = g_2 \scrT, \gamma g_1 \scrC = g_2 \scrC, \quad \gamma \gamma_1 g_1 \scrC =\gamma_2 g_2 \scrC.$$
Especially, we have
$$(\ga, g\scrC, g\scrT ) \sim ( \tilde{\ga}\ga \tilde{\ga}^{-1},  \tilde{\ga} g\scrC,  \tilde{\ga} g\scrT), \quad \forall  \tilde{\ga} \in \Ga.$$
We define the notion of a ``closed pointed straight strip of type $w$" in $\scrB_\Ga$ to be an equivalence class of $\scrP$.
Denote by $\Conv( \langle \ga \rangle g\scrC)$ the convex hull of  $\{\ga^n g \scrC\}_{n \in \Z}$ and we shall prove the following lemma.

\begin{lemma} \label{lemma32}
For $\ga\in \Ga$ and $g\in G$, suppose that $g^{-1} \ga g \in I w^k I$ for some {$k \in \Z$}. Then $\Conv( \langle \ga \rangle g\scrC)$ is a straight strip.
\end{lemma}

\begin{proof} First, let us consider the case that $g =e$, the identity element, and $\ga=w^k$.
Since $\ell(w^k) = k \ell(w)$,
the convex hull of  $\{w^{kn} \scrC\}_{-m \leq n \leq m}$ is equal to the convex hull of $w^{km}\scrC \cup w^{-km}\scrC$. Moreover, it is clear that when $m$ goes to infinity, the convex hull $\Conv( \langle w^k \rangle \scrC)$ is the whole straight strip $\scrT$.

Next, let us consider the general case. Since $g^{-1} \ga^{2n} g \in (IwI)^{2n} = I w^{2n} I$,
there exists some simplicial automorphism $h_n \in G$ which maps the pair $(w^{-n}\scrC, w^{n} \scrC)$ to $(\ga^{-n} g \scrC, \ga^n g \scrC)$. Especially, $h_n$ maps the convex hull $\Conv(w^{-kn} \scrC \cup w^{kn}\scrC)$ to the convex hull $\Conv((\ga^{-n} g \scrC \cup \ga^n g \scrC)$.

Since
$$ \Conv(\ga^{-n} g \scrC \cup \ga^n g \scrC) \subseteq  \Conv(\ga^{-m} g \scrC \cup \ga^m g \scrC)$$
when $n \leq m$,  and $\Conv(\ga^{-n} g \scrC \cup \ga^n g \scrC)\subset h_n \scrA$, it follows by a compactness argument that there exists some apartment $h_\infty \scrA$,
such that $ \bigcup_{n=1}^\infty  \Conv(\ga^{-n} g \scrC \cup \ga^n g \scrC) \subset h_\infty \scrA$.
In this case, we may further assume $\ga^n g \scrC = h_\infty w^{kn} \scrC$ for all $n \in \Z$. Then the convex hull
 $\Conv( \langle \ga \rangle g\scrC) =  h_\infty\Conv( \langle w^k \rangle \scrC) =  h_\infty\scrT$, is a straight strip.
\end{proof}

{Combing Lemma \ref{lemma31} and Lemma \ref{lemma32}, when $\gamma(g\scrT) = g\scrT$, $\Conv( \langle \ga \rangle g\scrC)$ is a straight strip which is also a subset of $g \scrT$. Thus we have the following proposition.
\begin{proposition} \label{prop1}
For $(\ga, g\scrC,g\scrT)\in \scrP$, $(\ga, g\scrC,g\scrT)= (\ga, g\scrC, \Conv( \langle \ga \rangle g\scrC)).$
\end{proposition} 
}
Now let $g_1\scrC, \cdots g_n \scrC$ be a complete list of liftings of chambers in $\scrB_\Gamma$.

\begin{theorem}
The set $\cup_{k=1}^\infty \scrP_k$ forms a set of representatives of the equivalence classes in $\scrP$, where
$$ \scrP_k=\{(\ga, g_i\scrC,  \Conv( \langle \ga \rangle g_i\scrC), \ga \in \Ga, i\in\{1,2, \ldots n\}, g_i^{-1} \ga g_i \in Iw^k I\}$$
\end{theorem}

\begin{proof} For  an element $p=(\ga, g\scrC, g\scrT)$ of $\scrP$, first we show that $p$ is equivalent to some element in $\scrP_k$.
Write $g\scrC = \delta g_i \scrC$ for some representative $g_i\scrC$ and $\delta \in \Ga$.
{Then by Proposition \ref{prop1}, 
$$p=  (\ga, g\scrC, \Conv( \langle \ga \rangle g\scrC))= 
(\ga, \delta g_i \scrC , \Conv( \langle \ga \rangle \delta g_i \scrC))$$} 
which is equivalent to 
$$p'=( \delta^{-1}\ga \delta, g_i \scrC , \delta^{-1} \Conv( \langle \ga \rangle \delta g_i \scrC)) = 
( \ga', g_i \scrC , \Conv( \langle \ga'
 \rangle  g_i \scrC))
$$
where $\ga'= \delta^{-1} \ga \delta \in \Ga$. Next we show that $p'$ is indeed an element of some $\scrP_k$. Since $g\scrC = \delta g_i \scrC$ and 
$ \ga(g\scrT)=g \scrT$,
we have $g a = \delta g_i $ for some $a \in I$ and 
$ g^{-1} \ga g \in I w^k I$ for some positive integer $k$.
Therefore, 
$$ g_i^{-1} \ga' g_i = g_i^{-1} \delta^{-1} \ga \delta g_i = a^{-1} g^{-1} \ga g a \in I w^k I.$$
Thus, $p' \in \scrP_k$.

To complete the proof, it remains to show that all elements in $\cup_{k=1}^\infty \scrP_k$ are not equivalent. Suppose that there are two equivalent elements $(\ga_1, g_{i_1}\scrC,  \Conv( \langle \ga_1 \rangle g_{i_1}\scrC)$,
$(\ga_2, g_{i_2}\scrC,  \Conv( \langle \ga_2 \rangle g_{i_2}\scrC))$. Then
$$ \ga g_{i_1} \scrT = g_{i_2} \scrT, \gamma g_{i_1} \scrC = g_{i_2} \scrC, \quad \gamma \gamma_1 g_{i_1} \scrC =\gamma_2 g_{i_2} \scrC$$
for some $\ga \in \Ga$. Since $\{g_i\scrC\}$ are representatives of $\Gamma$-orbits, we have $g=g_{i_1}=g_{i_2}$. Therefore $\ga g \scrC = g \scrC$ which implies $\ga \in \Ga \cap g I g^{-1} = \{ e\}$ (since the intersection is a discrete compact torsion-free subgroup). Thus, $\ga$ must be the identity element. Applying the third part of the above condition,
$$\gamma \gamma_1 g_{i_1} \scrC =\gamma_2 g_{i_2} \scrC \, \Rightarrow \, \gamma_1 g \scrC =\gamma_2 g \scrC
\, \Rightarrow \, \ga_1^{-1} \ga_2 \in \Ga \cap g \I g^{-1} \, \Rightarrow \,\ga_1 = \ga_2.
$$
We conclude that any two equivalent elements are always the same element.

\end{proof}

Now for a closed pointed straight strip $p$ represented by some element  in $\scrP_k$, we define its length $\ell(p)$ to be $k$.

\subsection{Counting closed pointed straight strips}
{Recall that one can identify the (affine) Hecke algebra $H_q(W,S)$ and the Iwahori-{Hecke} algebra $H(G,I)$ by mapping $e_w$ to $IwI$. Let $\pi_\Ga$ be the natural representation of $H_q(W,S)$ acting on $L^2(\Ga \backslash G/\I)$} and 
let $A_w$ be the matrix of $e_w= \I w \I$ under the basis of characteristic functions on $\Gamma g_1 I, \cdots,\Gamma g_n I$, denoted by $\{f_1,\cdots,f_n\}$. Write
$$ Iw^k I =  \bigsqcup_{\alpha \in \Omega_k} \alpha I,$$
Then 
$$ (A_w^k)_{i,j} =\# \{ \alpha \in \Omega_k, \Gamma g_i I = \Gamma g_j \alpha I\}.$$
Note that when $\Gamma g_i I = \Gamma g_j \alpha I$, there exists some $\ga \in \Ga$, such that
$ \ga g_i I = g_j \alpha I$ and such $\ga$ is unique since if there exists $\ga'$ satisfying the same condition, then again we have $\ga' \ga^{-1} \in \Gamma \cap g_i I g_i^{-1}=\{e\}$.
Therefore, we can rewrite the above as
$$ (A_w^k)_{i,j} = \# \{ (\alpha, \ga) \in \Omega_k \times \Ga,  \ga g_i I = g_j \alpha I\}.$$
On the other hand, when $\ga g_i I = g_j \alpha I$, we have $\ga g_i I \subset g_j I w^k I$ and conversely, when $\ga g_i I \subset g_j I w^k I$ there exists a unique $\alpha \in \Omega_k$ such that $\ga g_i I = g_j \alpha I$. Therefore, we obtain
$$ (A_w^k)_{i,j} = \# \{ \ga \in \Ga,  \ga g_i I \subseteq g_j Iw^k I\}
=\# \{ \ga \in \Ga,  g_j^{-1} \ga g_i \in  Iw^k I\}
.$$
Consequently, we have
 $$ \tr( (A_w)^k)=\# \{ (g_iI, \ga), i=1\sim n, \ga \in \Ga, g_i^{-1}\ga g_i  \in I w^k I.\} = \# \scrP_k.$$
We summarize the above discussion in the following theorem..

\begin{theorem}
The cardinality of $\scrP_k$ is equal to the trace of $(A_{w})^k$.
\end{theorem}

\subsection{Closed straight strips}
We shall regard closed pointed straight strips as analogues of closed geodesics with a fixed starting vertex in finite graphs. In this subsection, we define the concept of closed straight strips, which are an analogue of closed geodesics without a fixed starting vertex in finite graphs. Set
$$ \tilde{\scrP}=\{(\ga, g \scrT), g \in G, \gamma \in \Gamma, \gamma(g \scrT) = g \scrT , \ga\mid_{g \scrL} \mbox{is a positive translation}\}$$
and $(\ga_1, g_1 \scrT ) \sim (\ga_2, g_2 \scrT)$ if there exists some chambers $C_1$ and $C_2$ such that  $(\ga_1, C_1, g_1 \scrT ) \sim (\ga_2, C_2, g_2 \scrT)$.
Equivalence classes of $\tilde{\scrP}$  are called closed straight strips in $\scrB_\Ga$.

Now consider the canonical map $\sigma$ from $\scrP$ to $\tilde{\scrP}$ given by $\sigma((\ga, g\scrC, g \scrT)) = (\ga, g\scrT)$. By Lemma \ref{lemma31}, we have
$$ \sigma^{-1}((\ga, g\scrT)) = \{ (\ga, g w^k \scrC , g \scrT), k \in \Z\}.$$
Observe that the equivalence relation on $\tilde{\scrP}$ is induced form the equivalence relation on $\scrP$ and $\sigma$ induces a map $\tilde{\sigma}$ from $\scrP/\sim$ to $\tilde{\scrP}/\sim$, which is the map from the set of closed straight strips to the set of closed pointed straight strips by dropping pointed chambers. The size of the preimage of the equivalence class of $[(\ga, g\scrT)]$ is given by
$$  \# \tilde{\sigma}^{-1}( [(\ga, g\scrT)]) = \# \left( \{(\ga, g w^k \scrC , g \scrT), k \in \Z\}/\sim \right).$$
To evaluate the left-hand side of the above equation, recall that the setwise stabilizer $\Stab_\Ga(g\scrT)$ is a cyclic group containing $\ga$. Let $\ga_0$ be the unique generator of $\Stab_\Ga(g\scrT)$ satisfying
that $ \ga|_{g\scrL}$ is a positive translation. Then $\ga = (\ga_0)^m$ for some positive integer $m$ and
$\ga_0 g \scrC  = g w^{k_0} \scrC$ for some positive integer $k_0$.
Then
$$\left( \{(\ga, g w^k \scrC , g \scrT), k \in \Z\}/\sim \right) \, = \,\{(\ga, g w^k \scrC , g \scrT), k \in \Z\}/ \langle \ga_0 \rangle$$
which contains exactly $k_0$ elements. Note that one shall regard the closed strip $[(\ga, g \scrT)]$ {as} the closed strip $[(\ga_0, g \scrT)]$ repeated $m$-times in $\scrB_\Ga$. Besides, we call the closed straight strip represented by $(\ga_0, g \scrT)$ a primitive closed strip.
In the end, we summarize the above discussion as the following theorem.
\begin{theorem} \label{thm31}
Every closed straight strip is a repetition of some unique primitive closed straight strip. Moreover,
the number of closed pointed closed straight {strips} mapping to a given closed straight strip is equal to the length of its primitive closed straight strip.
\end{theorem}

\subsection{Zeta functions of straight strips}
Finally, for $i=1$ and 2, let us define the straight strip zeta function of type $w_i$ on $\scrB_\Ga$ as
$$ Z_{w_i}(\scrB_\Ga, u) = \exp\left( \sum_{n=1}^\infty \frac{N_n u^n}{n}\right)$$
where $N_n$ is the number of closed pointed straight strips of type $w_i$ of length $n$. Then we have
\begin{equation} \label{eq1}
Z_{w_i}(\scrB_\Ga, u) = \exp\left( \sum_{k=1}^\infty \frac{\tr((A_{w_i})^k)u^k}{k}\right) = \det(1- A_{w_i} u)^{-1}.
\end{equation}
One can also define the zeta function to be

$$ Z_{w_i}(\scrB_\Ga, u) = \prod_{p} (1- u^{l(p)})^{-1}$$
where $p$ runs through all closed primitive closed strips of type $w_i$. Then {by} the standard argument as the case of graph zeta function, Theorem \ref{thm31} implies these two definitions of zeta functions coincide.

{Finally, let $\rho = \pi_\Ga$ in the Corollary \ref{corollary1}. Note that by the definition of $H_i(\rho,u)$, we have  
$$ H_i(\rho,u)  = (I- A_{w_i} u^{\ell(w_i)})^{-1}.$$
Together with Eq.(\ref{eq1}), we have the following theorem.}

\begin{theorem}[Hashimoto-Ihara's formula for quotients of two dimensional buildings] \label{maintheorem2}
$$ \det \Alt(W)(\pi_\Ga,u)=Z_{w_1}(\scrB_\Ga,u^{\ell(w_1)})Z_{w_2}(\scrB_\Ga,u^{\ell(w_2)}). $$
\end{theorem}

\section{Alternating products of Poincar\'{e} series}
Note that one can generalize the definition of straight strips to ``straight tubes" in the higher rank cases in the following manner. Once again one identifies the apartment with a real vector space and the origin with a special vertex $v$ , and then for each vertex other than $v$ with a corresponding vector $v_i$ one may consider the connected components of the complement of the union of all the hyperplanes which are invariant under translation by the vector $v_i$. In this way one obtains the straight tubes in the direction $v_i$ for each $i$, and one may study their stabilizers. We expect that there exists an analogue of Theorem \ref{maintheorem1} and Theorem \ref{maintheorem2} for higher ranks cases. More precisely, we have the following conjecture.
\begin{conj}
For an affine Weyl group $(W,S)$ of rank $n$ (and $|S|=n+1)$, there exist $w_1,\cdots, w_n \in W$ such that for $H_i=\{ w_i^k\}_{k \in \Z_{\geq 0}}$,
the following two identities holds.
\begin{enumerate}
\item $w_i$ is the generator of the stabilizer of some straight tube, and $\ell(w_i^k)=k\cdot\ell(w_i)$ for all integers $k\geq 0$.
\item $\Alt(W)(u)=H_1(u)\cdots H_n(u)$.
\item For any representation $\rho$ of $H_q(W,S)$,
there exists certain ordering of
$$\bigg\{ H_1(\rho,u),\cdots,  H_n(\rho,u)\bigg\}\bigcup \bigg\{ W_{I}(\rho,u)^{(-1)^{|I|}}, I \subsetneq S\bigg\}$$
such that their product under this ordering is equal to $W(\rho,u)$.
\end{enumerate}
\end{conj}

Note that the above conjecture implies that
\begin{equation} \label{eq1}
\Alt(W)(u)\m = (1-u^{d_1}) \cdots (1-u^{d_n})
\end{equation}
where $d_i = \ell(w_i)$ are positive integers.

In the rest of the paper, we examine Eq. (\ref{eq1}) for  all affine Coxeter groups.


Let $R$ be an irreducible reduced crystallographic root system of rank $n$, and let $W$ be the Weyl group with generating set $S$, and let $\tilde{W}$ be the affine Weyl group with generating set $\tilde{S}$.
Let $h$ be the Coxeter number of $W$. The main result of this section is the following theorem.
\begin{theorem} \label{theorem2}
The power series $\Alt(\tilde{W})(u)\m$ is indeed a polynomial of the form
$$(1-u^{d_1})\cdots (1-u^{d_n})$$
where $d_i$ are integers with $n+1= d_1 \leq d_2 \leq \cdots \leq d_n \leq h$ as shown as in the following table.
\end{theorem}
\begin{center}
 \begin{tabular}{|c|l|}
 \hline
 Type &$d_1, \cdots, d_n$ \\ [0.5ex]
 \hline \hline
 ${A_{n}}$ & $n+1, n+1, n+1, \cdots, n+1$ \\ \hline
 ${B_{n}}$ & $n+1 \sim 2n$ \\ \hline
 ${C_{n}}$ & $n+1 \sim  2n$ \\ \hline
 ${D_{n}}$ & $n+1 \sim 2n-2, 2n-2, 2n-2$ \\ \hline
 ${E_{6}}$ & $7, 9, 9, 11, 12, 12$ \\ \hline
 ${E_{7}}$ & $8, 10, 11, 13, 14, 17, 18$ \\ \hline
 ${E_{8}}$ & $9, 11, 13, 14, 17, 19, 23, 29$ \\ \hline
 ${F_{4}}$ & $5, 7, 8, 11$ \\ \hline
 ${G_{2}}$ & $3, 5$ \\ \hline
 \end{tabular}
\end{center}

\subsection{MacDonald's formula}
MacDonald \cite{Mac} derived a formula for Poincar\'{e} series of Weyl groups and affine Weyl groups in terms of positive roots. Let us recall his result.
Let $B=\{a_1,...,a_n\}$ be the set of  simple roots of $R$ and $\rho$ be the highest root.
Let $\tilde{R}$ the associated affine root system, with base $\tilde{B}=B \cup \{a_0\}$, where $a_0$ = $1 - \rho$. For an affine root $a = \sum c_i a_i$, all $c_i$ are integers and  all of the same sign,
Its height is defined as
$$ ht(a)= \sum_{i=0}^n c_i.$$
Consider the set
$$ P=\{a \in \tilde{R}, 0 < ht(a) < h\}=R^+ \cup \{1-a, a \in R^+\}$$
where $h$ is the Coxeter number which is equal to $ht(1)$.
Let $W$ be the Weyl group and $\tilde{W}$ be the affine Weyl group of $R$ with the generating set $S$ and $\tilde{S}$ respectively.
\begin{theorem}
The following identities hold.
$$W(u)= \prod_{a \in R^+} \frac{1-u^{ht(a)+1}}{1-u^{ht(a)}} \qquad \mbox{and} \qquad \tilde{W}(u)= \frac{1}{(1-u^h)^n} \prod_{a \in P} \frac{1-u^{ht(a)+1}}{1-u^{ht(a)}}.$$
\end{theorem}
Now for an affine root $a = \sum c_i a_i$, define
$$ I(a)=\{a_i \in \tilde{B}, c_i \neq 0\}.$$
For a subset $I$ of $\tilde{B}$, let $R_I$ be the sub-root-system of $\tilde{R}$ with base $I$ and $W_I$ be its Weyl group (and $W_I=\tilde{W}$ if $R_I=\tilde{R}$).
By abuse of the notation, we also the denote the set of generators of $W_I$ by $I$, which is a subset of $\tilde{S}$.

Note that for any $R_J \supset R_{I(a)}$, $a$ is always a root in the sub root system $R_J$. (See Proposition 24, Chapter VI.7 of \cite{Bou}.)
Moreover, positive roots in any proper sub-root-sytem are contained in $P$.

\subsection{Alternating products}
For a subset $I$ of $S$, consider  the M\"{o}bius function
$ \mu(I) = (-1)^{|I|}$ which satisfies the property
$$ \sum_{I \subset J \subset S} \mu(J) =
\begin{cases}
\mu(S) &,\mbox{ if }I=S;\\
0&,\mbox{ otherwise}.
\end{cases}
$$
Applying Macdonald's formula, for the Weyl group $W$,  we have
\begin{align*}
\Alt(W)(u) &= \prod_{I\subset S} W_{I}(u)^{\mu(I)\mu(S)}\\
&= \prod_{I\subset S} \prod_{a \in R_I^+} \left(\frac{1-u^{ht(a)+1}}{1-u^{ht(a)}} \right)^{\mu(I)\mu(S)} \\
&=\prod_{a\in R^+}\prod_{I(a)\subset J \subset S}\left( \frac{1-u^{ht(a)+1}}{1-u^{ht(a)}} \right)^{\mu(J)\mu(S)}\\
&=\prod_{a\in R^+,I(a)= S} \frac{1-u^{ht(a)+1}}{1-u^{ht(a)}} \\
\end{align*}
Similarly, for the affine Weyl group $\tilde{W}$, we have
\begin{align*}
\Alt(\tilde{W})(u) &= \prod_{I\subset \tilde{S}} W_{I}(u)^{\mu(I)\mu(\tilde{S})}\\
&=   \frac{1}{(1-u^h)^n}  \prod_{a \in P} \left(\frac{1-u^{ht(a)+1}}{1-u^{ht(a)}} \right) \cdot  \prod_{I\subsetneq \tilde{S}} \prod_{a \in R_I^+} \left(\frac{1-u^{ht(a)+1}}{1-u^{ht(a)}} \right)^{\mu(I)\mu(\tilde{S})} \\
&= \frac{1}{(1-u^h)^n} \prod_{a\in P}\prod_{I(a)\subset J \subset \tilde{S}}\left( \frac{1-u^{ht(a)+1}}{1-u^{ht(a)}} \right)^{\mu(J)\mu(\tilde{S})}\\
&= \frac{1}{(1-u^h)^n}  \prod_{a\in P,I(a)= \tilde{S}} \frac{1-u^{ht(a)+1}}{1-u^{ht(a)}}\\
&= \frac{1}{(1-u^h)^n}  \prod_{a\in R^+,I(1-a)= \tilde{S}} \frac{1-u^{h-ht(a)+1}}{1-u^{h-ht(a)}}.
\end{align*}
We summarize the above result as the following theorem.
\begin{theorem} \label{theorem3}
The following identities hold.
$$\Alt(W)(u)= \prod_{a\in R^+,I(a)= S} \frac{1-u^{ht(a)+1}}{1-u^{ht(a)}} $$
and
$$\Alt(\tilde{W})(u)= \frac{1}{(1-u^h)^n}  \prod_{a\in R^+,I(1-a)= \tilde{S}} \frac{1-u^{h-ht(a)+1}}{1-u^{h-ht(a)}}.$$
\end{theorem}

\subsection{Sincere roots}
An affine root $a$  is called sincere in $R$  (resp. $\tilde{R}$) if $I(a)=B$ (resp. $I(a)=\tilde{B}$.)
One can easily find sincere roots from the complete list of positive roots (e.g. Appendix in \cite{Bou}) and obtain the following result.

\begin{center}
\begin{tabular}{|c|l|}
\hline
Type & Heights of sincere roots in $R^+$ \\
\hline
$A_n$ & $n$ \\
\hline
$B_n$ & $n \sim 2n-1$ \\
\hline
$C_n$ & $n \sim 2n-1$\\
\hline
$D_n$ & $n \sim 2n-3$\\
\hline
$E_6$& $6 \sim 8,8 \sim 11$\\
\hline
$E_7$& $7 \sim 13,9 \sim 17$\\
\hline
$E_8$& $8 \sim 19,10\sim 23, 12 \sim 29$\\
\hline
$F_4$& $4 \sim 7,6\sim 11$\\
\hline
$G_2$& $2 \sim 4$\\
\hline
\end{tabular}
\end{center}
Together with Theorem \ref{theorem3}, we obtain the following table.

\begin{center}
 \begin{tabular}{|c|l|}
 \hline
Type & $\Alt(W)(u)\m$ \\ [0.5ex]
 \hline
 $A_{n}$ & $\frac{1-u^{n+1}}{1-u^{n}}$ \\ \hline
 $B_{n}$ & $\frac{1-u^{2n}}{1-u^{n}}$ \\ \hline
 $D_{n}$ & $\frac{1-u^{2n-2}}{1-u^{n}}$ \\ \hline
 $E_{6}$ & $\frac{(1-u^{12})(1-u^{9})}{(1-u^{8})(1-u^{6})}$ \\ \hline
 $E_{7}$ & $\frac{(1-u^{18})(1-u^{14})}{(1-u^{9})(1-u^{7})}$ \\ \hline
 $E_{8}$ & $\frac{(1-u^{30})(1-u^{24})(1-u^{20})}{(1-u^{12})(1-u^{10})(1-u^{8})}$ \\ \hline
 $F_{4}$ & $\frac{(1-u^{12})(1-u^{8})}{(1-u^{6})(1-u^{4})}$ \\ \hline
 $G_{2}$ & $\frac{1-u^{5}}{1-u^{2}}$ \\ \hline
 \end{tabular}
 \end{center}
Next, consider sincere roots in $P$ which are of the form $1-a$ for some $a \in R^+$. By direct computation, one have
\begin{center}
\begin{tabular}{|c|c|l|}
\hline
Type & Coxeter number $h$ & heights of sincere roots in $P$ \\
\hline
$A_n$ & $n+1$ & no sincere roots\\
\hline
$B_n$ & $2n$ & $n+1 \sim 2n-1$ , $n+2 \sim 2n-1$, $\cdots$, $2n-1$ \\
\hline
$C_n$ &  $2n$ & $n+1 \sim 2n-1$ , $n+2 \sim 2n-1$, $\cdots$, $2n-1$ \\
\hline
$D_n$ &  $2n-2$ & $n+1 \sim 2n-3$ , $n+2 \sim 2n-3$, $\cdots$, $2n-3$ \\
\hline
$E_6$ & 12 & $7\sim11$, $9\sim11$, $9\sim11$, 11 \\
\hline
$E_7$ & 18 & $8\sim 17$, $10\sim 17$, $11\sim 17$, $13\sim 17$, $14\sim 17$, 17\\
\hline
$E_8$ & 30 & $9\sim 29$, $11\sim 29$, $13\sim 29$, $14\sim 29$, $17\sim 29$, $19\sim 29$, $23\sim 29$, 29 \\
\hline
$F_4$ & 12 & $5\sim 11$, $7 \sim 11$, $8 \sim 11$, 11 \\
\hline
$G_2$ & 6 & $3\sim 5$, 5 \\
\hline
\end{tabular}
\end{center}

Together with Theorem \ref{theorem3}, the proof of Theorem \ref{theorem2} is complete.

\end{document}